\documentclass[final,1p]{elsarticle}

\usepackage{lineno,hyperref}
\usepackage{amsmath,amsthm,amsfonts, amssymb}
\usepackage{mathtools}
\usepackage{algorithm}
\usepackage[noend]{algpseudocode}
\usepackage{subfigure}
\usepackage{color}
\usepackage{graphicx}
\usepackage{caption}
\usepackage{float}

\newcommand{\T}{\textrm{T}}

\modulolinenumbers[5]

\DeclareMathOperator*{\argmin}{arg\,min}

\newtheorem{theorem}{Theorem}[section]
\newtheorem{remark}{Remark}[section]
\newtheorem{lemma}{Lemma}[section]
\newtheorem{assumption}{Assumption}

\makeatletter
\newcommand*{\rom}[1]{\expandafter\@slowromancap\romannumeral #1@}
\makeatother

\begin{document}

\graphicspath{{figures/}}

\begin{frontmatter}

\title{Stochastic Gradient Descent for Semilinear Elliptic Equations with Uncertainties}


%
%
%
%

\author[add1]{Ting Wang}
\author[add1]{Jaroslaw Knap}

\address[add1]{Physical Modeling \& Simulation Branch, CISD, CCDC U.S. Army Research Laboratory}

\begin{abstract}
Randomness is ubiquitous in modern engineering. The uncertainty is
often modeled as random coefficients in the differential equations
that describe the underlying physics.  In this work, we describe a
two-step framework for numerically solving semilinear elliptic partial
differential equations with random coefficients: 1) reformulate
the problem as a functional minimization problem based on the
direct method of calculus of variation; 2) solve the minimization
problem using the stochastic gradient descent method.  We
provide the convergence criterion for the resulted stochastic gradient
descent algorithm and discuss some useful technique to overcome the issues of ill-conditioning and large variance. The accuracy and efficiency of the algorithm are demonstrated by numerical experiments.
\end{abstract}


\end{frontmatter}



\section{Introduction}
Many problems in science and engineering involve spatially varying
input data. Often, the input data is subject to uncertainties due to
inherent randomness. For example, the details of spatial variations of
properties and structure of engineering materials are typically
obscure, and randomness and uncertainty are fundamental features of
these complex physical systems. Under these circumstances, traditional
deterministic models are rarely capable of properly handling this
randomness and yielding accurate predictions. Therefore, in order to
furnish accurate predictions, randomness must be incorporated directly
into the model and the propagation of the resulting uncertainty
between its input and output must be quantified accordingly.
A model, particularly important in many applications, consists of the
random input data in the form of a random field and a partial
differential equation (PDE). The specific model problem considered
here is
\begin{equation}\label{eqn:abstract-PDE}
\mathcal{L}(\kappa) (u) = 0 \quad\textrm{in}~ D 
\end{equation}
where $\mathcal{L}$ is a nonlinear elliptic operator dependent on a
random field $\kappa$ over a probability space $(\Omega, \mathcal{F}, \mathbb{P})$, and $u$ is a
solution. Examples of~\eqref{eqn:abstract-PDE} include, but are not
limited to, flow of water through random porous medium and modeling of
the mechanical response of materials with random microstructure.

Over the past few decades, the critical need to seek solutions
of~\eqref{eqn:abstract-PDE} has yielded a wealth of numerical
approaches. Numerical methods for solving PDEs with random
coefficients have been traditionally classified into three major
categories: stochastic collocation (SC)~\cite{babuvska2007stochastic,
  nobile2008sparse, nobile2008anisotropic}, stochastic Galerkin (SG)
\cite{ghanem2003stochastic, babuska2004galerkin,
  gunzburger2014stochastic, xiu2002wiener, xiu2003modeling} and Monte
Carlo (MC) \cite{babuska2004galerkin, matthies2005galerkin,
  kuo2016application}.  SC aims to first solve the deterministic
counterpart of~\eqref{eqn:abstract-PDE} on a set of collocation points
and then interpolate over the entire image space of the random
element. Hence, the method is non-intrusive meaning that it can take
advantage of existing legacy solvers developed for deterministic
problems.  Similarly, MC is non-intrusive as well since it relies on
taking sample average over a set of deterministic solutions computed
from a set of realizations of the random field.  In contrast, SG is considered as
intrusive since it requires construction of discretizations of both
the stochastic space and physical space simultaneously and, as a
result, it commonly tends to produce large systems of algebraic
equations whose solutions are needed. However, these algebraic systems
are considerably different from their deterministic counterparts and
thus deterministic legacy solvers cannot be easily utilized.

We emphasize that all of the three categories discussed so far depend
on the stochastic weak formulation of~\eqref{eqn:abstract-PDE}.  In
this work, alternatively, we take the variational viewpoint and
reformulate problem~\eqref{eqn:abstract-PDE} as a problem of seeking
minimizers of the following functional
\begin{equation}\label{eqn:abstract-energy}
E(u) =  \mathbb{E}\left\{\int_D I(x, u, \nabla u, \omega) \, dx\right\},
\end{equation}
whose Euler-Lagrange equations coincide with~\eqref{eqn:abstract-PDE}
under suitable assumptions.  Therefore, over an appropriate space, 
solving~\eqref{eqn:abstract-PDE} is equivalent to
minimizing~\eqref{eqn:abstract-energy}. The use of variational
formulations has become widespread in many areas of science and
engineering due to their many
advantages~\cite{struwe1990variational,reddy2017energy}. First and
foremost, the equations in the weak form are often applicable in
situations when the strong form may no longer be valid. A case in
point is modeling of microstructure evolution in materials, where fine
scale oscillations may emerge, leading to highly irregular
solutions~\cite{ball1989fine,muller1999variational}. Second,
variational formulations are known to be remarkably convenient for
numerical computation as they often produce numerical methods capable
of preserving, at least to some extent, the structure of the original
problem~\cite{simo2006computational,marsden2001discrete}.

In order to establish the existence of minimizers of $E(u)$ in an
appropriate space $W$, by the direct method of calculus of variations
\cite{dacorogna2007direct}, it is sufficient to identify a minimizing
sequence $\{u_{\nu}\} \subset W$ which satisfies two properties:
\begin{enumerate}
\item[(\rom{1})] the sequence $\{u_{\nu}\}$ is compact under the weak topology on $W$, i.e., \[ u_{\nu} \rightharpoonup u^* ~\textrm{in}~W.\]
This is often implied by the boundedness of the sequence (up to the extraction of a subsequence), i.e., $\|u_{\nu}\|_W \leq \gamma$ for some constant $\gamma$ independent of $\nu$.
\item[(\rom{2})] the functional $E$ is lower semicontinuous with respect to weak convergence, i.e., 
\[u_{\nu} \rightharpoonup u^* ~\textrm{in}~W\quad \textrm{implies}\quad \liminf_{\nu\to \infty}E(u_{\nu}) \geq E(u^*).\]
\end{enumerate}
Given the above two properties, it is straightforward to verify that the function $u^*$ is indeed a minimizer of $E(u)$.
The direct method is not only of theoretical importance.  
From the numerical point of view, it suggests that
if we can identify a minimizing sequence $\{u_{\nu}\}$
each of which solves~\eqref{eqn:abstract-energy} over 
a finite dimensional subspace $W_{\nu} \subset W$, i.e., 
\[u_{\nu} = \argmin_{u \in W_{\nu}} E(u),\]
then the above two properties ensure that $u_{\nu}$ converges weakly to the minimizer $u^*$.
That is, when $\{u_{\nu}\}$ is interpreted as a sequence of solutions to~\eqref{eqn:abstract-energy} over a sequence of finite dimensional spaces $\{W_{\nu}\}$ that approximates $W$, the direct method of variational calculus automatically guarantees the numerical consistency.
Bearing this in mind, numerical approximation to~\eqref{eqn:abstract-energy} boils down to 
minimizing $E(u)$ over finite dimensional spaces $W_{\nu}$ with suitable optimization methods.

The fundamental difficulty in solving the stochastic optimization
problem~\eqref{eqn:abstract-energy} is that the expectation often
involves high dimensional integral which generally cannot be computed
with high accuracy~\cite{nemirovski2009robust}.  Thus, conventional
nonlinear optimization techniques are seldom suitable for problems
like~\eqref{eqn:abstract-energy} since an inaccurate gradient
estimation is usually detrimental to the convergence of the
algorithms.  In contrast, stochastic gradient descent (SGD) replaces
the actual gradient by its noisy estimate, but is guaranteed to
converge under mild
conditions~\cite{bottou2010large,bottou2018optimization,kingma2014adam}.
The method can be traced back to the Robbins--Monro
algorithm~\cite{robbins1951stochastic} and has nowadays become one of
the cornerstone for large-scale machine
learning~\cite{bottou2018optimization}.  However, due to the noisy
nature of SGD iteration, a naive use of the algorithm in many
instances suffers difficult tuning of parameters and extremely slow
convergence rate~\cite{nemirovski2009robust}.
In this article, we describe an application of SGD to construct
numerical schemes for the solution of the variational stochastic
problem~\eqref{eqn:abstract-energy}. We also provide simple, yet
powerful, strategies for efficient and robust SGD algorithms in the
above context.

The reminder of the article is organized as follows. In
Section~\ref{sec:model}, we setup the semilinear model problem and
impose several running assumptions on the model.  The variational
reformulation of the model problem as a stochastic minimization
problem is described in
Section~\ref{sec:direct-method-and-PCE}. Afterward, in
Section~\ref{sec:SGD}, we propose to utilize the SGD to solve the
minimization problem and discuss some useful technique for noise
reduction and convergence acceleration for SGD. Finally, numerical
benchmarks are presented in Section~\ref{sec:example}.
 
\section{Model problem}\label{sec:model}
We introduce a probability space $(\Omega, \mathcal{F}, \mathbb{P})$
where $\Omega$ is the set of all events, $\mathcal{F}$ is the
$\sigma$-algebra consisting of all measurable events and
$\mathbb{P}$ a probability measure.  We consider the following
semilinear elliptic PDE with random coefficient $\kappa$ defined in $(\Omega,
\mathcal{F}, \mathbb{P})$,
\begin{equation}\label{eqn:PDE}
\begin{split}
-\nabla \cdot (\kappa(x, \omega) \nabla u(x, \omega)) + f(x, u(x, \omega), \omega) & = 0 \qquad x \in D\\
\qquad u(x, \omega) &= 0 \qquad x \in \partial D , 
\end{split}
\end{equation}
where the domain $D$ is a bounded subset of $\mathbb{R}^d$, the boundary $\partial D$ is either smooth or convex and piecewise smooth, the diffusion coefficient $\kappa: D \times  \Omega \to \mathbb{R}$ is a random field with continuous and bounded covariance functions and the nonlinear term $f: D \times \mathbb{R} \times \Omega$ is sufficiently smooth for almost surely all $\omega \in \Omega$.
We assume the solution to~\eqref{eqn:PDE} exists and is unique.

We deal with the case of finite dimensional noise, i.e., 
there exists finitely many independent random variables $Y_1, \ldots, Y_K$ such that
\[
\kappa(x, \omega) = \kappa(x, Y_1(\omega), \ldots, Y_K(\omega)).
\]
These random variables are often referred as the stochastic germ that bring randomness into the system.
From the practical perspective, the finite dimensional noise assumption is reasonable since
the random input often
admits a parametrization in terms of finitely many random variables.
From the theoretical perspective, by the Karhunen-Lo\`eve (KL) expansion~\cite{mercer1909xvi}:
when the random field $\kappa(x, \omega)$ is square integrable with continuous covariance function, i.e., \[
\kappa(x, \cdot) \in L_{\mathbb{P}}^2(\Omega), \qquad \forall x \in D 
\]
and the covariance function
\[
\textrm{Cov}_{\kappa}(x, y) = \mathbb{E}[(\kappa(x) - \mathbb{E}[\kappa(x)])(\kappa(y) - \mathbb{E}[\kappa(y)])]
\]
is a well-defined continuous function of $x, y \in D$, then the random field $\kappa$ can be approximated by the truncated KL expansion
\[
\kappa(x, \omega) \approx \bar{\kappa}(x) + \sum_{n=1}^N  \sqrt{\lambda_n} \psi_n(x) Y_n(\omega),
\]
where $\bar{\kappa}$ is the mean of the random field, 
$(\lambda_n, \psi_n)$ are eigen-pairs of the covariance kernel 
\[(C_{\kappa} g)(x) \triangleq \int_D \textrm{Cov}_{\kappa}(x, y) g(y)\, dy\]
and $Y_n(\omega)$ are uncorrelated and identically distributed random variables 
with mean zero and unit variance.

Consequently, when the randomness of~\eqref{eqn:PDE} is completely characterized by finitely many independent
random variables ${\bf Y} = (Y_1, \ldots, Y_K)$,
problem~\eqref{eqn:PDE} is equivalent to 
\begin{equation}\label{eqn:PDE-finite-noise}
\begin{split}
-\nabla \cdot \left(\kappa(x, {\bf Y}) \nabla u(x, {\bf Y})\right) + f(x, u(x, {\bf Y}), {\bf Y}) &= 0, \qquad x \in D,\\
u(x, {\bf Y}) &= 0, \qquad x \in \partial D
\end{split}
\end{equation}
by the Doob-Dynkin’s lemma~\cite{oksendal2013stochastic}.  To define a
suitable space of solution of the above problem, we introduce the
physical space $V = \mathcal{H}_0^1(D)$, i.e., the Sobolev space of
functions with weak derivatives up to order $1$ and vanishing on the
boundary. We also define the stochastic space $S = L_{\mathbb{P}}^2(\Omega)$, i.e.,
the space of $\mathbb{R}$-valued square integrable (with respect to $\mathbb{P}$)
random variables.  
The solution $u(x, {\bf Y})$ to \eqref{eqn:PDE-finite-noise} is now thus defined in the tensor space 
\[
V \otimes S = \mathcal{H}_0^1(D) \otimes L_{\mathbb{P}}^2(\Omega)
= \left\{v: D \times \Omega \to \mathbb{R} \left|~ v \in \mathcal{B}(D)\otimes \mathcal{F}, \mathbb{E}\left\{\|v\|_{\mathcal{H}_0^1(D)}^2\right\} < \infty\right.\right\}.
\]

Now we make some technical assumptions on $\kappa$ and $f$.
We denote $\bar{D}$ the closure of $D$.
\begin{assumption}\label{assume:kappa}
$\kappa(x, {\bf Y})$ is uniformly bounded and uniformly coercive, i.e.,
there exist constants $0 < \kappa_{\min} \leq \kappa_{\max}$ such that
\[\mathbb{P}\left(\omega \in \Omega : \kappa_{\textrm{min}} \leq \kappa(x, {\bf Y}(\omega)) \leq \kappa_{\textrm{max}}, ~\forall x \in \bar{D} \right) = 1.\]
\end{assumption}
\begin{assumption}\label{assume:f}
$f(x, u, {\bf Y})$ is uniformly bounded, i.e., there exists a constant $f_{\textrm{max}} > 0$
such that
\[
\mathbb{P}\left(\omega \in \Omega : |f(x, u, {\bf Y}(\omega))| \leq  f_{\textrm{max}}, ~\forall x \in \bar{D}, \forall  u \in \mathbb{R}  \right) = 1.
\]
Furthermore, $f(x, u, {\bf Y}(\omega))$ is uniformly Lipschitz continuous in $u$, i.e., there exists a Lipschitz constant $L_f > 0$ such that
\[\mathbb{P}\left(\omega \in \Omega : |f(x, u_1, , {\bf Y}(\omega)) - f(x, u_2, , {\bf Y}(\omega))| \leq L_f | u_ 1 - u_2 |,~ \forall u_1, u_2 \in \mathbb{R},  \forall x \in \bar{D} \right) = 1.\]
Finally,  $\partial_u f(x, u, {\bf Y})$ is uniformly bounded from below, i.e., there exists a constant $\delta > 0$ such that 
\[
\mathbb{P}(\omega \in \Omega :  \partial_u f(x, u, {\bf Y}(\omega)) \geq \delta,
~ \forall x \in \bar{D}, \forall u \in \mathbb{R}) = 1.
\]
\end{assumption}
As we shall see hereafter, these
assumptions are crucial to guarantee the convergence of the SGD
algorithm.


\section{Direct method and polynomial chaos expansion}\label{sec:direct-method-and-PCE}
\subsection{Direct method in calculus of variations}\label{subsec:direct-method}
Following basic ideas of the calculus of variation, 
our starting point is to reformulate the stochastic PDE problem~\eqref{eqn:PDE-finite-noise}
as the minimization problem
\begin{equation}\label{eqn:exact-energy}
\min_{u \in V \otimes S} E(u) = \mathbb{E}\left\{\int_D \frac{1}{2} \kappa(x, {\bf Y}) |\nabla_x u(x, {\bf Y})|^2 + F(x, u(x, {\bf Y}), {\bf Y}) \, dx\right\}
\end{equation}
with $\partial_u F(x, u, y) = f(x, u, y)$, and the expectation $\mathbb{E}$ taken with respect to the random vector ${\bf Y}$.
We first show that~\eqref{eqn:exact-energy} has a minimizer and this minimizer satisfies the weak form of~\eqref{eqn:PDE-finite-noise}. The result is a simple application of the direct method in variational calculus \cite{dacorogna2007direct}.
\begin{theorem}\label{thm:existence-and-uniqueness}
Under Assumptions~\ref{assume:kappa} and~\ref{assume:f} and additionally that
$F$ is uniformly bounded from below, then the problem~\eqref{eqn:exact-energy} has a minimizer $u^* \in V\otimes S$.
Furthermore, $u^*$ satisfies the following weak form of~\eqref{eqn:PDE-finite-noise}:
\begin{equation}\label{eqn:weak-form}
    \mathbb{E}\left\{\int_D \kappa(x,{\bf Y}) \nabla_x u(x, {\bf Y}) \cdot \nabla_x v(x, {\bf Y}) + f(x, u, {\bf Y}) v(x, {\bf Y}) \, dx\right\}=0, \quad \forall v \in V \otimes S.
\end{equation}
\end{theorem}

\begin{proof}
As we sketched in the introduction, it is sufficient to show that 1) there exists a minimizing sequence $u_{\nu}(x, {\bf Y})$ that converges weakly to some $u^* \in V\otimes S$ and 2) the functional $E(u)$ is (weakly) lower semicontinuous~\cite{dacorogna2007direct}.

For weak convergence of the sequence $u_{\nu}(x, {\bf Y})$, it suffices to uniformly bound $u_{\nu}$ under the $V \otimes S$ norm.
To this end, note that by Assumption~\ref{assume:kappa} and the uniformly lower boundedness of $F$
\[
\frac{1}{2}\kappa(x, {\bf Y}) |\nabla_x u_{\nu}(x, {\bf Y})|^2 + F(x, u_{\nu}, {\bf Y}) \geq \frac{1}{2} \kappa_{\text{min}}|\nabla_x u_{\nu}(x, {\bf Y})|^2 + F_{\textrm{min}}.
\]
Integrating over $D$ and taking expectation of both sides lead to
\[
E(u_{\nu}) \geq \frac{1}{2} \kappa_{\text{min}}\mathbb{E}\left\{\int_D |\nabla_x u_{\nu}(x, {\bf Y})|^2 \, dx  \right\} + \int_D F_{\textrm{min}} \, dx.
\]
Invoking the Poincar\'{e}'s inequality, there exist some constants $C_1>0$ and $C_2\geq 0$ such that
\[
E(u_{\nu}) \geq C_1 \mathbb{E}\{\|u_{\nu}\|_{\mathcal{H}_0^1}^2 \} + C_2.
\]
Now note that $E(u_{\nu})$ is uniformly bounded (in $\nu$) since it is a minimizing sequence of~\eqref{eqn:exact-energy}, which implies $\mathbb{E}\{\|u_{\nu}\|_{\mathcal{H}_0^1}^2 \}$ is uniformly bounded and hence there exists $u^* \in {V \otimes S}$ such that 
\[u_{\nu} \rightharpoonup u^* \quad \textrm{in}~{V \otimes S}.\] 

Next, we justify that $E(u)$ is (weakly) lower semicontinuous. To this end,
we define
\[
E_1(u) = \mathbb{E}\left\{\int_D \frac{1}{2} \kappa(x, {\bf Y}) |\nabla_x u(x, {\bf Y})|^2 \, dx\right\}
\]
and 
\[
E_2(u) = \mathbb{E}\left\{\int_D F(x, u, {\bf Y}) \, dx\right\}
\]
and show both $E_1$ and $E_2$ are lower semicontinuous.
To see the lower semicontinuity of $E_1(u)$, 
since $u_{\nu}$ converges weakly to $u^*$ in the Hilbert space $V \otimes S$, by definition (through choosing the test function $\kappa \nabla_x u^*/2$)
\[
E_1(u^*)=
\lim_{\nu \to \infty}
\mathbb{E}\left\{\int_D \frac{1}{2}\kappa(x, {\bf Y}) \nabla_x u^*(x, {\bf Y}) \cdot  \nabla_x u_{\nu}(x, {\bf Y}) \,dx \right\}.
\]
Squaring both sides and then applying the Cauchy-Schwarz inequality lead to
\[
E_1(u^*)^2 \leq E_1(u^*)\liminf_{\nu \to \infty} E_1(u_{\nu}) .
\]
If $E_1(u^*) > 0$ (the case when $E_1(u^*) = 0$ is trivial since $E_1(u_{\nu}) \geq 0$), the above inequality implies 
\[
E_1(u^*) \leq \liminf_{\nu \to \infty} E_1(u_{\nu}),
\]
i.e., $E_1(u)$ is (weak) lower semicontinuous.
Now for $E_2(u)$, by Taylor expansion in $u$ and the uniform boundedness of $\partial_u F = f$ (Assumption~\ref{assume:f}), we have
\[
\mathbb{E}\left\{\int_D F(x, u_{\nu}, {\bf Y})\, dx\right\}
-
\mathbb{E}\left\{\int_D F(x, u^*, {\bf Y})\, dx\right\}
\leq
f_{\textrm{max}} \mathbb{E}\left\{\int_D |u_{\nu} - u^*| \, dx   \right\}.
\]
The lower semicontinuity of $E_2(u)$
follows immediately from the weak convergence of $u_{\nu}$ to $u^*$.
Therefore, $E(u)$ is lower semicontinuous and hence 
has a minimizer $u^*$. 

It remains to be shown that $u^*$ satisfies the weak form~\eqref{eqn:weak-form}.
To this end, we consider the functional $E$ evaluated at $u^* + \epsilon v$ for every $v \in V \otimes S$, i.e., $E(u^* + \epsilon v)$. 
A simple calculation shows that the Gateaux derivative satisfies
\begin{equation*}
\begin{split}
\lim_{\epsilon \to 0}
\frac{1}{\epsilon}(E(u^* + \epsilon v) - E(u^*))
=&
\mathbb{E}\left\{\int_D \kappa(x, {\bf Y}) \nabla_x u^* \cdot \nabla_x v \, dx\right\}\\
&+ \lim_{\epsilon \to 0} \mathbb{E}\left\{\int_D f(x, u^* + \epsilon(x, {\bf Y}) v, {\bf Y}) v\, dx\right\}
\end{split}
\end{equation*}
for some random variable $\epsilon(x, {\bf Y}) \in (-\epsilon,\epsilon)$.
Since $f$ is uniformly bounded by Assumption~\ref{assume:f}, by the dominated convergence theorem
\[
\lim_{\epsilon \to 0} \mathbb{E}\left\{\int_D f(x, u^* + \epsilon(x, {\bf Y}) v, {\bf Y})v\, dx\right\} = \mathbb{E}\left\{\int_D f(x, u^*, {\bf Y})v \, dx\right\}.
\] 
Owing to the fact that $u^*$ is a
minimizer, the Gateaux derivative
\[
\left.\frac{d}{d\epsilon}\right|_{\epsilon = 0}E(u^* + \epsilon v) = 0,
\]
which yields the weak form~\eqref{eqn:weak-form}.
\end{proof}


Due to the infinite dimensionality of the solution space $V \otimes S$, it is not practical to solve~\eqref{eqn:exact-energy} directly and hence we seek an approximate solution $u_{\nu}$ over a finite dimensional subspace $(V\otimes S)_{\nu} \subset V\otimes S$, where $\nu$ is a generic index parameterizing the approximation accuracy. Consequently, the approximated minimization problem over the finite dimensional subspace becomes
\begin{equation*}
\min_{u \in (V\otimes S)_{\nu}} E(u) = \mathbb{E}\left\{\int_D \frac{1}{2} \kappa(x, {\bf Y}) |\nabla_x u(x, {\bf Y})|^2 + F(x, u(x, {\bf Y}), {\bf Y}) \, dx \right\}.
\end{equation*}
Denote $u_{\nu}$ the minimizer of $E(u)$ over the finite dimensional subspace $(V\otimes S)_{\nu}$. 
Since $u_{\nu}$ is a minimization sequence of~\eqref{eqn:exact-energy}, i.e., 
\[E(u_{\nu}) \to E(u^*)\quad \textrm{as} \quad \nu\to \infty,\]
the sequence $\{u_{\nu}\}$ is
compact on $V \otimes S$ by Theorem~\ref{thm:existence-and-uniqueness}, that is, 
$u_{\nu}$ converges weakly to $u^*$, a minimizer of~\eqref{eqn:exact-energy}.
That is to say, the consistency of the numerical approximation is guaranteed automatically under the framework of direct method of variational calculus.
This serves as the theoretical foundation of the algorithm proposed in this work. Thus, the problem is reduced to finding a finite dimensional subspace minimizer $u_{\nu}$ as an approximation to the infinite dimensional space minimizer $u^*$.


\subsection{Polynomial chaos expansion}
In this section, we make the above finite dimensional subspace approximation
$(V\otimes S)_h$ explicit.
It is customary to construct approximations in the physical space $V$ by means
of polynomials, in particular those with compact support, as is the case for the finite-element method. Undeniably, there exist various finite dimensional approximations to the stochastic space $S$.
For the sake of clarity, throughout this paper we adopt the generalized polynomial chaos (PC) expansion~\cite{xiu2002wiener} as a convenient  approximation method in space $S$. 
However, we emphasize that the general framework presented in this work extends naturally to other approaches for constructing approximants, such as piecewise polynomials expansions~\cite{ghanem2003stochastic}, multiwavelet decompositions~\cite{le2004a, le2004b} to name a few.

The PC expansion is essentially a representation of second order random objects (e.g., random variables in $L_{\mathbb{P}}^2(\Omega)$ or stochastic fields in $L^2(D) \otimes L_{\mathbb{P}}^2(\Omega)$)~\cite{xiu2002wiener, xiu2003modeling, le2010spectral}.
Given a random variable $X$ in $(\Omega, \mathcal{F}, \mathbb{P})$, the PC expansion asserts that we can identify a set of $L_{\mathbb{P}}^2(\Omega)$-orthogonal univariate polynomial bases $\{\psi_j\}$ so that any function $r : \mathbb{R} \to \mathbb{R}$ satisfying  $r(X) \in L_{\mathbb{P}}^2(\Omega)$ can be expressed as
\[r(X)  =  \sum_{j =0}^{\infty} r_j \psi_j(X)\]
in the $L_{\mathbb{P}}^2(\Omega)$ sense, where the coefficients are
\[r_j = \frac{\mathbb{E}\{r(X) \psi_j(X)\}}{\mathbb{E}\{\psi_j^2(X)\}}, \qquad j = 0, 1, \ldots.\]
For instance, $\{\psi_j\}$ are Hermite polynomials when $X$ is normal and are Legendre polynomials when $X$ is uniform.
The PC expansion can be generalized to the case of $K$ dimensional random vector 
${\bf X} = (X_1, \ldots, X_K)$ with independent components.
Specifically, for any function $R : \mathbb{R}^K \to \mathbb{R}$ satisfying  $R(X) \in L_{\mathbb{P}}^2(\Omega)$ we can write
\[R({\bf X}) = \sum_{j=0}^{\infty} R_j \Psi_j(X),\]
where $\{\Psi_j\}$ are $K$-variate polynomials involving products of those univariate polynomials associated with each component.

Now, given the $K$ dimensional random vector ${\bf Y}$ in~\eqref{eqn:exact-energy}, 
we can expand the random field $u(x, {\bf Y}) \in V \otimes S$ as a generalized PC series
 \[u(x, {\bf Y}) = \sum_{j=0}^{\infty} u_j(x)\Psi_j({\bf Y}),\]
 where the coefficient function
 \[
u_j(x) = \frac{\mathbb{E}\{u(x, {\bf Y}) \Psi_j({\bf Y})\}}{\mathbb{E}\{\Psi_j^2({\bf Y})\}}, \qquad \forall j = 0, 1, \ldots. 
 \]  
 In practice, truncation of the PC series is required for numerical approximation. 
 To this end, we define $S_N$, the finite dimensional subspace spanned by the $L_{\mathbb{P}}^2(\Omega)$-orthogonal polynomials $\{\Psi_0({\bf Y}), \ldots, \Psi_N({\bf Y})\}$, i.e.,
\[
S_N = \textrm{span}\{\Psi_0({\bf Y}), \ldots, \Psi_N({\bf Y})\}.
\]
Note that $N$ is determined by both the stochastic dimensionality $K$ and the highest order of the basis polynomials $p$ through
\[N + 1 =  \frac{(p + K)!}{p! K!}.\]
Hence, the dimensionality of the stochastic subspace can be very high when $p$ and $K$
are large.
In order to further expand the coefficients $u_j(x)$, we approximate it over  
\[
V_M = \textrm{span}\{\phi_1, \ldots, \phi_M\} \subset V,
\]
the finite dimensional subspace of $V$ spanned by the bases $\{\phi_1, \ldots, \phi_M\}$.
Therefore, the finite dimensional subspace of $V \otimes S$ is $V_M \otimes S_N$ over which we have a finite dimensional approximation
\[u_c(x, {\bf Y}) \triangleq \sum_{i=1}^M \sum_{j=0}^N c_{ij} \phi_i(x) \Psi_j({\bf Y}) \approx u(x, {\bf Y}).\]
Note that in the notation of $u_c$ we omit the dependence on $M$ and $N$ in order to simplify the notation.
For convenience, we define the vector valued function $\Gamma(x, {\bf Y})$ 
consisting of all bases of $V_M \otimes S_N$ with the following numbering of index
\[
\Gamma(x, {\bf Y}) =
(\phi_1(x) \Psi_0({\bf Y}), \ldots, \phi_M(x) \Psi_0({\bf Y}), \ldots \ldots,  \phi_1(x) \Psi_N({\bf Y}), \ldots, 
\phi_M(x) \Psi_N({\bf Y}))^{\T}.
\]
Hence, the approximated solution $u_c$ can be written in the following compact form
\begin{equation}\label{eqn:intrusive-PCE}
u_c(x, {\bf Y})  =  c^{\T}\Gamma(x, {\bf Y}),
\end{equation}
where the coefficient (column) vector $c$ is
\[c = [c_{1, 0}, \ldots, c_{M, 0}, \ldots\ldots, c_{1, N}, \ldots, c_{M, N}]^{\T} \in \mathbb{R}^{M(N+1)}.\]
Over the finite dimensional subspace $V_M \otimes S_N$, the functional associated with $u_c$ becomes
\begin{equation}\label{eqn:approx-energy}
E(u_c) = \mathbb{E}\left\{\int_D \frac{1}{2} \kappa(x, {\bf Y}) |\nabla_x u_c(x, {\bf Y})|^2  + F(x, u_c(x, {\bf Y}), {\bf Y}) \, dx \right\}.
\end{equation}
Note that $E(u_c)$ is indeed a function of $c$, hence we rewrite 
$E(u_c)$ as a function of the coefficients $c$, denoted by $J(c)$. 
Therefore, minimizing $E(u_c)$ is equivalent to minimizing the following function with respect to the coefficient vector $c$,
\begin{equation}\label{eqn:min-problem}
\min_{c \in \mathbb{R}^{M(N+1)}} J(c) = \min_{c \in \mathbb{R}^{M(N+1)}} J^1(c) + J^2(c),
\end{equation}
where
\begin{equation}
\begin{split}
J^1(c) &= \mathbb{E}\left\{\int_D \frac{1}{2}\kappa(x, {\bf Y}) \left| c^{\T} \nabla_x\Gamma(x, {\bf Y})\right|^2 \, dx \right\},\\
J^2(c) &= \mathbb{E} \left\{ \int_D F(x, c^{\T} \Gamma(x, {\bf Y}), {\bf Y})\, dx \right\}
\end{split}
\end{equation}
are the values associated with the linear part
and the nonlinear part of~\eqref{eqn:PDE-finite-noise}, respectively.
Here the gradient of $\Gamma(x, {\bf Y})$ with respect to $x$ is defined as  
\[
\nabla_x \Gamma(x, {\bf Y}) = \left[\nabla_{x_1} \Gamma(x, {\bf Y}), \ldots, \nabla_{x_d} \Gamma(x, {\bf Y}) \right] \in \mathbb{R}^{M(N+1) \times d}.
\]

Finally, we make the following technical assumption regarding the basis functions. 
\begin{assumption}\label{assume:basis}
For each $i = 1, \ldots, M$, the physical-space basis satisfies
\[ \int_D |\phi_i(x)|^4 \, dx < \infty, \qquad \int_D | \nabla_x \phi_i(x) |^2 \, dx < \infty.\]
For each $j = 0, \ldots, N$, the PC basis satisfies
\[\mathbb{E} \{|\Psi_j({\bf Y})|^4\} < \infty.\]
\end{assumption} 
The integrability conditions are satisfied in most settings. For example, when the physical domain $D$ is compact and $\phi_i$ are finite element bases, the integrability is readily verified. For the stochastic space, when ${\bf Y}$ is normal or uniform, $\Psi_j({\bf Y})$ has finite moments of all orders.

\section{Stochastic gradient descent for semilinear problem}\label{sec:SGD}
\subsection{Convergence of stochastic gradient descent}
As mentioned in Section~\ref{subsec:direct-method}, in order to find an approximation to the solution of problem~\eqref{eqn:PDE}, it is sufficient to solve the stochastic optimization problem~\eqref{eqn:min-problem}.
The natural choice for optimizing the function $J(c)$ is the stochastic gradient descent~\cite{robbins1951stochastic}, which is one of the most fundamental ingredients of large-scale machine learning~\cite{bottou2010large, bottou2018optimization}. 
In this section, we discuss an application of SGD to the specific minimization problem~\eqref{eqn:min-problem}.
For convenience, we denote the unbiased estimator of $\nabla J(c)$ by 
\begin{equation}\label{eqn:gradient}
g(c, {\bf Y}) = g^1(c, {\bf Y}) + g^2(c, {\bf Y}),
\end{equation}
where
\begin{equation*}
\begin{split}
g^1(c, {\bf Y}) &= \left(\int_D \kappa(x, {\bf Y}) \nabla_x\Gamma(x, {\bf Y}) \nabla_x\Gamma(x, {\bf Y})^{\T}   \, dx\right) c \in \mathbb{R}^{M(N+1)}\\
g^2(c, {\bf Y}) &= \int_D f(x, c^{\T} \Gamma(x, {\bf Y}), {\bf Y}) \Gamma(x, {\bf Y})\, dx \in \mathbb{R}^{M(N+1)}
\end{split}
\end{equation*}
so that $\mathbb{E} \{g(c, {\bf Y})\} = \nabla J(c)$.
Instead of computing the deterministic gradient $\nabla J(c_n)$ at each iteration, SGD simply requires the 
stochastic gradient $g(c_n, {\bf Y}_n)$ for each iteration:  
\begin{equation}\label{eqn:vanilla-SGD}
c_{n+1} = c_n - \eta_n g(c_n, {\bf Y}_n), \qquad n \geq 1,
\end{equation}
where $\eta_n$ is the learning rate.
In the context of machine learning, each ${\bf Y}_n$ corresponds to a randomly picked example from the dataset.
In the context of this article, we interpret ${\bf Y}_n$ as a realization of
the stochastic germ ${\bf Y}(\omega_n)$.
Note that the iterative sequence of coefficients $\{c_n\}$ is a sequence of random variables since each $c_n$ depends on ${\bf Y}_1, \ldots, {\bf Y}_{n-1}$ for $n \geq 1$.

Intuitively, SGD works because, while each direction $-g(c, {\bf Y})$ may not be one of the descent directions of $c$, it is, however, a descent direction in expectation.
It is clear that SGD is advantageous as it only requires computation of a single realization of the gradient at each iteration.
Yet, it is a fundamental question whether SGD applied to the problem~\eqref{eqn:min-problem}
produces a convergent sequence minimizing the function $J(c)$.
To answer this question, we first present three important lemmas concerning the properties of the function $J(c)$ and the gradient estimator $g(c, {\bf Y})$.
\begin{lemma}\label{lem:Lipschitz-gradient}
Under Assumptions~\ref{assume:kappa}, \ref{assume:f} and~\ref{assume:basis}, 
the function $J(c):\mathbb{R}^{M(N+1)} \to \mathbb{R}$ is continuously differentiable 
and its gradient $\nabla J(c) : \mathbb{R}^{M(N+1)} \to \mathbb{R}^{M(N+1)}$ is Lipschitz continuous 
with Lipschitz constant $L > 0$, i.e., for any $c_1, c_2 \in \mathbb{R}^{M(N+1)}$,
\[
| \nabla J(c_1) - \nabla J(c_2)|  \leq L | c_1 - c_2 |. 
\]
\end{lemma}
\begin{proof}
Throughout the proof, $L$ denotes a generic positive constant that may differ by a scaling constant.  
Let $c_1, c_2 \in \mathbb{R}^{M(N+1)}$ be two arbitrary vectors of coefficients for the expansions over $V_M \otimes S_N$.
For the linear part of $J(c)$, by Assumption~\ref{assume:f}, 
\begin{equation*}
\begin{split}
\left|\nabla J^1(c_1) - \nabla J^1(c_2) \right|^2  
&=
\left|\mathbb{E}\left\{ \int_D \kappa(x, {\bf Y}) \nabla_x \Gamma(x, {\bf Y}) \nabla_x \Gamma(x, {\bf Y})^{\T} \, dx \right\}
 (c_1 - c_2) \right|^2\\
 &\leq
 L \left|\mathbb{E}\left\{ \int_D  \nabla_x \Gamma(x, {\bf Y}) \nabla_x \Gamma(x, {\bf Y})^{\T} \, dx \right\}
 (c_1 - c_2) \right|^2\\
 &\leq
 L \left\|\mathbb{E}\left\{ \int_D  \nabla_x \Gamma(x, {\bf Y}) \nabla_x \Gamma(x, {\bf Y})^{\T}  \, dx \right\}
 \right\|^2  |c_1 - c_2|^2,
\end{split}
\end{equation*}
where $\|\cdot\|$ is the matrix $2$-norm. 
To see that the matrix norm $\left\|\mathbb{E}\left\{ \int_D  \nabla_x \Gamma(x, {\bf Y}) \nabla_x \Gamma(x, {\bf Y})^{\T}  \, dx \right\}
 \right\|^2$ is finite, 
it is sufficient to bound terms of the form 
\begin{equation}\label{eqn:linear-Lipschtiz}
 \left(\int_D \partial_{x_{k_1}} \phi_{i_1}(x)  \partial_{x_{k_2}} \phi_{i_2}(x)  \, dx  \right)
\mathbb{E} \left\{  \Psi_{j_1}({\bf Y}) \Psi_{j_2}({\bf Y}) \right\}    
\end{equation}
with $k_1, k_2 = 1, \cdots, d$, $i_1, i_2 = 1, \cdots, M$ and $j_1, j_2 = 0, \cdots, N$. 
By Assumption~\ref{assume:basis}, 
the finiteness of~\eqref{eqn:linear-Lipschtiz} follows immediately.

Now for the nonlinear part $J^2(c)$, 
by Jensen's inequality and Assumption~\ref{assume:f}
\begin{equation}
\begin{split}
\left|\nabla J^2(c_1) - \nabla J^2(c_2) \right|^2  
&=
\left|\mathbb{E}\left\{\int_D (f(x, u_{c_1}, {\bf Y}) - f(x, u_{c_2}, {\bf Y})) \Gamma(x, {\bf Y}) \, dx\right\}\right|^2\\
&\leq
L\mathbb{E} \left\{  \int_D |u_{c_1} - u_{c_2}|^2
|\Gamma(x, {\bf Y})|^2 \, dx    \right\} \\
&\leq 
L \mathbb{E}\left\{ \int_D |\Gamma(x, {\bf Y})|^4   \, dx \right\} | c_1 - c_2 |^2.
\end{split}
\end{equation}
Finally, $\mathbb{E}\left\{ \int_D |\Gamma(x, {\bf Y})|^4   \, dx \right\}$ is finite by Assumption~\ref{assume:basis}.
\end{proof}

\begin{lemma}\label{lem:2nd-moment-bound}
Under Assumptions~\ref{assume:kappa}, \ref{assume:f} and~\ref{assume:basis}, 
there exist constants $M_1 \geq 0$ and $M_2 > 0$ such that
\[
\mathbb{E}\{ | g(c, {\bf Y}) |^2\} \leq M_1 | \nabla J(c)   |^2 + M_2 .
\]
That is, the second moment of the gradient estimator is allowed to grow quadratically
in the mean gradient.
\end{lemma}
\begin{proof}
In virtue of Assumption~\ref{assume:kappa} and~\ref{assume:f},
\begin{equation*}
\begin{split}
\mathbb{E}\{ | g(c, {\bf Y})|^2\} 
&\leq 2 \mathbb{E}\{| g^1(c, {\bf Y})|^2\} + 2 \mathbb{E}\{| g^2(c, {\bf Y})|^2\}\\
&\leq 2 \kappa_{\textrm{max}}^2 \mathbb{E} \left\{ \left\|  \int_D \nabla_x \Gamma(x, {\bf Y}) \nabla_x \Gamma(x, {\bf Y})^{\T} \,dx  \right\|^2   \right\} | c |^2\\
&\qquad + 2 f_{\textrm{max}}^2 \mathbb{E} \left\{ \left|\int_D  \Gamma(x, {\bf Y})  \, dx \right|^2  \right\},
\end{split}
\end{equation*}
which is finite by Assumption~\ref{assume:basis}.
Similar to the proof of Lemma~\ref{lem:Lipschitz-gradient}, 
we can readily show that $\| \nabla J(c) \|^2$ is finite as well. 
Therefore, we can choose two appropriate constants $M_1, M_2 > 0$ such that
\[
\mathbb{E}\{ | g(c, {\bf Y})|^2\} \leq M_1 | \nabla J(c) |^2+ M_2.
\]
\end{proof}

\begin{lemma}\label{lem:strongly-convex}
Under Assumptions~\ref{assume:kappa}, \ref{assume:f} and an additional assumption that
for all $x \in D$ and almost surely all $\omega \in \Omega$, $f(x, c^T \Gamma, {\bf Y}(\omega))$, when viewed as a function of $c$, is continuously differentiable with respect to $c$, i.e., for all $x \in D$, 
\[\mathbb{P}\left(\omega \in \Omega : f(x, c^T \Gamma, {\bf Y}(\omega))~\textrm{continuously differentiable w.r.t.}~c\right) = 1.\]
Then, the function $J(c) : \mathbb{R}^{M(N+1)} \to \mathbb{R}$ is strongly convex (in $c$), i.e., there exists a constant $\lambda > 0$, such that for any $c_1, c_2 \in \mathbb{R}^{M(N+1)}$
\[
(\nabla J(c_1) - \nabla J(c_2) )^{\T} (c_1 - c_2) \geq \lambda |c_1 - c_2 |^2
\]
and hence $J(c)$ has a unique minimizer $c^*$.
\end{lemma}

\begin{proof}
Note that by Assumption~\ref{assume:kappa},
\begin{equation}\label{eqn:strong-convex-estimate-1}
\begin{split}
(\nabla J^1(c_1) - \nabla J^1(c_2))^{\T}(c_1 - c_2)
\geq 
\kappa_{\textrm{min}} (c_1 - c_2)^{\T}\mathbb{E}\left\{ \int_D \nabla_x \Gamma \nabla_x \Gamma^{\T} \,dx  \right\} (c_1 - c_2).
\end{split} 
\end{equation} 
In view of the continuously differentiability of $f$ in $c$, there exists $\widetilde{c} = t c_1 + (1-t)c_2$ for some $t \in [0, 1]$ (that may depend on $x$) such that
\begin{equation*}
\begin{split}
\nabla J^2(c_1) - \nabla J^2(c_2)
&=
\mathbb{E}\left\{ \int_D (f(x, c_1^{\T} \Gamma, {\bf Y}) - f(x, c_2^{\T} \Gamma, {\bf Y})) \Gamma \, dx\right\}\\
&= 
\mathbb{E}\left\{ \int_D \nabla_c f(x, \widetilde{c}^{\T} \Gamma, {\bf Y})^{\T} (c_1 - c_2) \Gamma \, dx\right\}\\
&=
\mathbb{E}\left\{ \int_D \partial_u f(x, \widetilde{c}^{\T} \Gamma, {\bf Y}) \Gamma \Gamma^{\T}  (c_1 - c_2)  \, dx    \right\}.
\end{split}
\end{equation*}
Since $\partial_u f$ is uniformly bounded from below by Assumption~\eqref{assume:f}, we have
\begin{equation}\label{eqn:strong-convex-estimate-2}
\begin{split}
(\nabla J^2(c_1) - \nabla J^2(c_2))^{\T}(c_1 - c_2) 
& = (c_1 - c_2)^{\T}\mathbb{E}\left\{ \int_D  \partial_u f(x, \widetilde{c}^{\T} \Gamma, {\bf Y}) \Gamma  \Gamma^{\T} \, dx    \right\}(c_1 - c_2)\\
&\geq \delta (c_1 - c_2)^{\T}\mathbb{E}\left\{ \int_D  \Gamma  \Gamma^{\T} \, dx    \right\}(c_1 - c_2).
\end{split}
\end{equation}
The strong convexity follows immediately by combining~\eqref{eqn:strong-convex-estimate-1} and~\eqref{eqn:strong-convex-estimate-2}.
\end{proof}

Now, we are ready to present the main result concerning the convergence of SGD when applied for solving problem~\eqref{eqn:min-problem}.
Given Lemmas~\ref{lem:Lipschitz-gradient}, \ref{lem:2nd-moment-bound} and~\ref{lem:strongly-convex},
the proof of the following result is simply a straightforward application of Theorem~$4.7$ in~\cite{bottou2018optimization}.
\begin{theorem}\label{thm:main-thm}
Under the same assumptions as in Lemmas~\ref{lem:Lipschitz-gradient}, \ref{lem:2nd-moment-bound} and~\ref{lem:strongly-convex},
suppose that the learning rate $\eta_n$ satisfies
\[
\eta_n = \frac{\beta}{\gamma + n},   \qquad n > 1
\]
for some constants $\beta > \lambda^{-1}$ and $\gamma > 0$ such that $\eta_1 \leq L^{-1}M_1^{-1}$.
Then, the function $J(c)$ decays sublinearly to the minimum in expectation., i.e., 
\[
\mathbb{E}\{J(c_n)\} - J(c^*) \leq \frac{\nu}{\gamma + n},
\]
where $\nu = \dfrac{\beta^2 L M_2}{2(\beta \lambda - 1)}$.
\end{theorem}

\begin{remark}
Some remarks about the above result are in order. 
\begin{enumerate}
\item Diminishing learning rate has to be used to guarantee the convergence. The initial
learning rate $\eta_1$ cannot be larger than a certain threshold.
\item We comment that the $\mathcal{O}(n^{-1})$ rate is the fastest convergence rate that the stochastic gradient descent can achieve \cite{agarwal2012information}. However, the multiplicative constant can be improved by incorporating the second order information of the function $J(c)$. We will discuss more on this aspect in the next chapter.
\item A similar result can be obtained when the function $J(c)$ is convex but not strongly convex. However, the convergence deteriorates to $\mathcal{O}(n^{-1/2})$ \cite{nemirovski2009robust}. 
\end{enumerate}
\end{remark}

It is well known that the SGD~\eqref{eqn:vanilla-SGD} suffers from the adverse effect of noisy gradient estimation.
On the other hand, there is no particular reason to estimate the gradient only based on one realization of the random variables at each iteration. 
Therefore, it is natural to introduce a mini-batch at each iteration in order to ``stabilize"
the algorithm.
That is, at the $n$-th iteration we average the gradient over a batch of $N_g$ realizations of the random variable ${\bf Y}$ in order to obtain a less noisy gradient estimation
\[
g_{\textrm{mb}}(c_n, {\bf Y}_n^{\textrm{mb}}) = \frac{1}{N_g} \sum_{i=1}^{N_g} g(c_n, {\bf Y}_{n, i}),
\]
where ${\bf Y}_n^{\textrm{mb}} = ({\bf Y}_{n,1}, \ldots, {\bf Y}_{n,{N_g}})$
and each 
${\bf Y}_{n,i}$ is the $i$-th realization in the mini-batch at the $n$-th iteration of SGD. 
With the mini-batch gradient, we have the mini-batch SGD
\begin{equation}\label{eqn:mini-batch-SGD}
c_{n+1} = c_n - \eta_n g_{\textrm{mb}}(c_n, {\bf Y}_n^{\textrm{mb}}).
\end{equation}
Clearly, the mini-batch averaging reduces the variance of gradient estimation by a factor of $1 / N_g$ but is also $N_g$ times more expensive than standard SGD.
More sophisticated mini-batch strategies can be applied to accelerate SGD~\cite{zhao2014accelerating, dekel2012optimal}. 
Our simulation results suggest that mini-batch is crucial for the convergence of SGD in order to take a relatively large learning rate.

\subsection{The second order SGD}
Further improvements to the SGD~\eqref{eqn:vanilla-SGD} may be
achieved by way of incorporating the second order information
pertaining to the function $J(c)$.  Theorem~\ref{thm:main-thm}
elucidates the fact that the constant appearing in the
$\mathcal{O}(n^{-1})$ convergence rate depends on $\nu$, which in turn
depends on the condition number of the Hessian $\nabla^2 J(c)$.  This
is similar to the deterministic optimization where the second order
information is often incorporated to overcome the ill-conditioning of
the optimization problem~\cite{nocedal2006numerical}.  The same
approach can be utilized in the stochastic setting by means of
adaptively rescaling the stochastic gradients based on matrices
capturing local curvature information of the function $J(c)$, so that
the constant is significantly improved as a result.  More precisely,
we consider an iteration scheme
\begin{equation}\label{eqn:mini-batch-second-order-SGD}
c_{n+1} = c_n - \eta_n H_n g_{\textrm{mb}}(c_n, {\bf Y}_n^{\textrm{mb}}), \qquad n \geq 1,
\end{equation}
where $H_n$ is a symmetric positive definite approximation to the inverse of the Hessian $(\nabla^2 J(c_n))^{-1}$.
In fact, it was shown in \cite{bottou2005line} that if $H_n$ is updated dynamically such that $H_n \to (\nabla^2 J(c^*))^{-1}$, then the multiplicative constant appearing in the $\mathcal{O}(n^{-1})$ convergence rate is independent of the condition number of the Hessian.
It is true that approximation of $H_n$ is often based on a small set of samples and hence is very noisy. 
However, it has been long observed that the Hessian matrix need not be as accurate as the gradient in order to yield an effective iteration since
the iteration $\eqref{eqn:mini-batch-second-order-SGD}$ is more tolerant to noise in the Hessian estimate than it is to noise in the gradient estimate.
Therefore, it may be beneficial to incorporate partial Hessian information in the stochastic setting. 
To this end, we denote the unbiased estimator of the Hessian $\nabla^2 J(c)$ as
\begin{equation}\label{eqn:Hessian}
h(c, {\bf Y}) = h^1(c, {\bf Y}) + h^2(c, {\bf Y}),
\end{equation}
where
\begin{equation*}
\begin{split}
h^1(c, {\bf Y}) &= \int_D \kappa(x, {\bf Y}) \nabla_x\Gamma(x, {\bf Y})    \nabla_x\Gamma(x, {\bf Y})^{\T} \, dx  \in \mathbb{R}^{M(N+1)\times M(N+1)},\\
h^2(c,{\bf Y}) &= \int_D  \partial_u f(x, c^{\T} \Gamma(x, {\bf Y}), {\bf Y}) \Gamma(x, {\bf Y}) \Gamma(x, {\bf Y})^{\T} \, dx \in \mathbb{R}^{M(N+1)\times M(N+1)}
\end{split}
\end{equation*}
so that $\mathbb{E}\{h(c, {\bf Y})\} = \nabla^2 J(c)$. 
Note that the estimator $h^1(c, {\bf Y})$ is indeed independent of $c$ whereas $h^2(c, {\bf Y})$ depends on $c$ in a nonlinear way. 
For convenience, we define two $M$ by $M$ matrices $A({\bf Y})$ and $B(c, {\bf Y})$ with
components
\[
A_{i_1 i_2}({\bf Y}) = \int_D \kappa(x, {\bf Y})\left(\partial_{x_1} \phi_{i_1}(x) \partial_{x_1}\phi_{i_2}(x)  
+ \ldots + \partial_{x_d} \phi_{i_1}(x) \partial_{x_d}\phi_{i_2}(x)\right) \, dx
\]
and 
\[
B_{i_1 i_2}(c, {\bf Y}) = \int_D   \partial_u f(x, c^{\T} \Gamma(x, {\bf Y}), {\bf Y}) \phi_{i_1}(x) \phi_{i_2}(x)    \, dx 
\]
for $i_1, i_2 = 1, \ldots, M$.
Then, we can express the Hessian estimator in the following block matrix form,
\begin{equation*}
h^1(c, {\bf Y}) = 
\begin{bmatrix}
A({\bf Y}) \Psi_0({\bf Y}) \Psi_0({\bf Y}) & \ldots & A({\bf Y}) \Psi_0({\bf Y}) \Psi_N({\bf Y})\\
\vdots & \ddots & \vdots  \\
A({\bf Y}) \Psi_N({\bf Y}) \Psi_0({\bf Y}) & \ldots & A({\bf Y}) \Psi_N({\bf Y}) \Psi_N({\bf Y})
\end{bmatrix}
\end{equation*}
\begin{equation*}
h^2(c, {\bf Y}) = 
\begin{bmatrix}
B(c, {\bf Y}) \Psi_0({\bf Y}) \Psi_0({\bf Y}) & \ldots & B(c, {\bf Y}) \Psi_0({\bf Y}) \Psi_N({\bf Y})\\
\vdots & \ddots & \vdots  \\
B(c, {\bf Y}) \Psi_N({\bf Y}) \Psi_0({\bf Y}) & \ldots & B(c, {\bf Y}) \Psi_N({\bf Y}) \Psi_N({\bf Y})
\end{bmatrix}.
\end{equation*}
At each iteration of~\eqref{eqn:mini-batch-second-order-SGD},
we attempt to obtain $H_n$ from limited realizations of the estimator $h(c_n, {\bf Y})$.
It is worthwhile noting that for each realization of ${\bf Y}$, the above two matrices are rank deficient. 
Nonetheless, their expectations with respect to ${\bf Y}$, i.e., $\mathbb{E}\{h^1(c, {\bf Y})\}$
and $\mathbb{E}\{h^1(c, {\bf Y})\}$, 
are of full rank given that the Hessian $\nabla^2 J(c)$ is invertible.
This fact suggests that for a robust second order scheme~\eqref{eqn:mini-batch-second-order-SGD}, we need to sample sufficient realizations of $h(c, \bf Y)$ at each iteration so that the sample average of the Hessian becomes invertible.
Thus, classical second order methods, such as BFGS~\cite{nocedal2006numerical},
have to be used with caution in this case.
On the other hand, employing the full second order information may not be necessary
considering the high computational cost (due to large size of the Hessian matrix) and the noisy nature of Hessian estimation.
Moreover, poor curvature estimation may even have an adverse effect to the convergence of SGD.
Motivated by the fact that diagonal matrices are often utilized as preconditioners to combat the ill conditioning issue in computational linear algebra,
we simply use the block diagonal of $h(c, \bf Y)$ as the scaling matrix, i.e., 
\begin{equation*}
H(c, {\bf Y}) = \textrm{block-diag}(h(c, {\bf Y})^{-1}).
\end{equation*}
It is important to note that the block diagonal matrix $H(c, {\bf Y})$ defined above is always of full rank. 
Multiplication of the gradient by $H$ is equivalent to applying a linear transformation to each segment of the search direction $g(c, {\bf Y})$ separately.
Finally, we are ready to present the SGD-PCE algorithm.

\begin{algorithm}[!ht]
\caption{The SGD-PCE algorithm}\label{alg:SGD-PCE}
\textbf{Input}: The total number of iterations $N_{\textrm{sgd}}$; the mini-batch size $N_g$ for gradient estimation and the mini-batch size $N_h$ for Hessian estimation\\
\textbf{Output}: The solution coefficients $c$
\begin{algorithmic}[1]
\For{$n = 1 : N_{\textrm{sgd}}$}
  \State Generate the random vector ${\bf Y}_n^g = ({\bf Y}_{n, 1}, \ldots, {\bf Y}_{n, N_g})$ for $n$-th mini-batch;
  \For{$i = 1 : N_g$}
    \State Compute the stochastic gradient $g(c_n, {\bf Y}_{n,i})$ as defined in~\eqref{eqn:gradient}
  \EndFor
  \State Compute the mini-batch (averaged) gradient 
  \[\bar{g}(c_n, {\bf Y}_n^g) = \frac{1}{N_g} \sum_{i=1}^{N_g}  g(c_n, {\bf Y}_{n, i})\]
  \State Generate the random vector ${\bf Y}_n^h = ({\bf Y}_{n, 1}^h, \ldots, {\bf Y}_{n, N_h}^h)$ for Hessian estimation
    \For{$j = 1 : N_h$}
    \State Compute the stochastic block diagonal Hessian $h_b(c_n, {\bf Y}_{n,j}^h))$
  \EndFor
  \State Compute the mini-batch (averaged) block diagonal Hessian
  \[\bar{h}_b(c_n, {\bf Y}_n^h) = \frac{1}{N_h}  \sum_{j=1}^{N_h} h_b(c_n, {\bf Y}_{n, j}^{h})\]
  \State Update $c_n \gets c_n - \eta_n \bar{h}_b(c_n, {\bf Y}_n^h)^{-1} \bar{g}(c_n, {\bf Y}_n^g) $
\EndFor
\end{algorithmic}
\end{algorithm}

\begin{remark}
A few remarks regarding the algorithm are in order. 
\begin{enumerate}
\item Since $h_{\textrm{mb}}$ is a block diagonal matrix, we only need to solve $N+1$
linear systems  each of size $M$ by $M$ in order to update $c_n$.
Hence, the the total complexity for this update is of complexity $\mathcal{O}(M^3 (N+1))$
whereas a full use of the second order information requires $\mathcal{O}(M^3(N+1)^3)$ operations.
\item In this algorithm, mini-batch size is fixed and the preconditioner is simply the block diagonal of the Hessian estimator. We believe that more sophisticated ideas can be applied  
to improve the algorithm further.
For instance, second order methods such as the natural gradient estimation~\cite{amari1998natural, amari2007methods}  
and a dynamic size mini-batch algorithm such as the 
dynamic batch size algorithm~\cite{bottou2018optimization} may be combined for more efficient algorithm.

\end{enumerate}
\end{remark}



\subsection{Variance reduction based on control variates}\label{subsec:control-variates}
In this section, we discuss further variance reduction by taking advantage of the special structure of the gradient estimator $g(c, {\bf Y})$~\eqref{eqn:gradient} and the Hessian estimator $h(c, {\bf Y})$~\eqref{eqn:Hessian}.
Specifically, we explore the application of the control variates (CV) technique~\cite{glasserman2013monte} for reducing the variances in gradient and Hessian estimations.
The fundamental idea of control variates can be briefly described as follows. 
Suppose we aim to estimate the mean of a given random variable $X$ and we have another auxiliary random variable $Z$ with known mean $\mathbb{E}Z = \mu$.
Then, we can construct a new random variable 
\[
\widetilde{X} = X + \lambda (Z - \mu), \qquad \lambda \in \mathbb{R}
\] 
such that $\mathbb{E}X = \mathbb{E}\widetilde{X}$.
That is, $\widetilde{X}$ is an unbiased alternative of $X$.
Minimizing the variance of $\widetilde{X}$ with respect to the parameter $\lambda$ gives the optimal value 
\[
\lambda^* = -\frac{\textrm{Cov}(X, Z)}{\textrm{Var}(X)}
\]
and the corresponding variance of $\widetilde{X}$ is
\begin{equation}\label{eqn:control-variates-variance}
\textrm{Var}(\widetilde{X})
= \textrm{Var}(X)\left[1 - \frac{\textrm{Cov}(X, Z)^2}{\textrm{Var}(X)\textrm{Var}(Z)}\right]
=  \textrm{Var}(X)\left(1 - \rho(X, Z)^2 \right),
\end{equation}
where 
\[\rho(X, Z) = \frac{\textrm{Cov}(X, Z)}{\sqrt{\textrm{Var}(X)\textrm{Var}(Z)}}\]
is the correlation between $X$ and $Z$. 
Hence, the variance reduction achieved through replacing $X$ by $\widetilde{X}$ depends on the choice of the auxiliary random variable $Z$. 
Note that $\rho(X, Z) \in [-1, 1]$ and as a consequence variance reduction is guaranteed so long as $X$ and $Z$ are correlated.
The best, but most likely unfeasible, scenario would amount to choosing $Z$ so that $\rho(X, Z) = 1$ and $\widetilde{X}$ becoming a zero variance estimator as a result.
Therefore, the key for designing an efficient CV random variable $\widetilde{X}$ is to select an auxiliary random variable $Z$ tightly coupled with the original random variable $X$.
We comment that $\lambda^*$ is seldom known, however, a coarse estimation of it  
from its empirical average is often sufficient. 

We apply now the CV strategy to estimate the linear part of the gradient, i.e., $g^1(c, {\bf Y})$.
Recall that $g^1(x, {\bf Y})$ involves elementary random terms 
\[
\kappa(x, {\bf Y}) \Psi_{j_1}({\bf Y}) \Psi_{j_2}({\bf Y}), 
\]
for all $j_1, j_2 = 0, \ldots, N$.
Owing to the fact that the moments of $\{\Psi_j(\bf Y)\}$ can often be computed in advance, we
consider the CV alternatives of the above estimators.
First, in the same spirit as the perturbation method~\cite{liu1986random, yamazaki1988neumann}, we approximate the random field $\kappa(x, {\bf Y})$ at $\mathbb{E}\{\bf Y\}$, i.e.,
\begin{equation}\label{eqn:kappa-zeroth-order-approx}
    \kappa(x, {\bf Y}) \approx \widetilde{\kappa}_0(x, {\bf Y}) \triangleq \kappa(x, \mathbb{E}\{{\bf Y}\}).
\end{equation}
Next, we can construct an unbiased alternative for $\kappa(x, {\bf Y}) \Psi_{j_1}({\bf Y}) \Psi_{j_2}({\bf Y})$
by using the auxiliary random variable $Z = \widetilde{\kappa}_0(x, {\bf Y}) \Psi_{j_1}({\bf Y})\Psi_{j_2}({\bf Y})$, i.e., 
\begin{equation*}
\kappa(x, {\bf Y})\Psi_{j_1}({\bf Y})\Psi_{j_2}({\bf Y}) + \lambda_{j_1, j_2}^* 
\left(
\widetilde{\kappa}_0(x, {\bf Y}) \Psi_{j_1}({\bf Y})\Psi_{j_2}({\bf Y}) - \mathbb{E}\{  \widetilde{\kappa}_0(x, {\bf Y}) \Psi_{j_1}({\bf Y})\Psi_{j_2}({\bf Y})  \}
\right),
\end{equation*}
where 
\[
\lambda_{j_1,j_2}^* = -\frac{\textrm{Cov}\left(\kappa(x, {\bf Y})\Psi_{j_1}({\bf Y})\Psi_{j_2}({\bf Y}), \widetilde{\kappa}_0(x, {\bf Y})\Psi_{j_1}({\bf Y})\Psi_{j_2}({\bf Y})    \right)}{\textrm{Var}\left( \kappa(x, {\bf Y})\Psi_{j_1}({\bf Y})\Psi_{j_2}({\bf Y})   \right)}.
\]
Note that $\mathbb{E}\{\widetilde{\kappa}_0(x, {\bf Y})\Psi_{j_1}({\bf Y})\Psi_{j_2}({\bf Y})\}$ is known since it only involves moments of $\{\Psi_j(\bf Y)\}$ which are, as mentioned before, ordinarily pre-computed.
Also note that, when $\kappa(x, {\bf Y})$ is of small uncertainty, $\widetilde{\kappa}_0(x, {\bf Y})\Psi_{j_1}({\bf Y})\Psi_{j_2}({\bf Y})$ is an approximation to 
$\kappa(x, {\bf Y})\Psi_{j_1}({\bf Y})\Psi_{j_2}({\bf Y})$. Hence, we expect that they are highly correlated and 
the variance reduction is significant by the relation~\eqref{eqn:control-variates-variance}.

Further variance reduction can be achieved by higher order approximation to $\kappa$.
For example, we can employ the first order approximation instead of zeroth order approximation,
i.e.,
 \begin{equation}\label{eqn:kappa-first-order-approx}
 \kappa(x, {\bf Y}) \approx \widetilde{\kappa}_1(x, {\bf Y}) \triangleq \kappa(x, \mathbb{E}\{{\bf Y}\}) + \nabla_{{\bf y}}\kappa(x, \mathbb{E}\{{\bf Y}\})^T  {\bf Y},
 \end{equation}
 where 
 $\nabla_{\bf y}\kappa(x, \mathbb{E}\{\bf Y\})$ is the gradient of $\kappa$ with respect to ${\bf y}$ evaluated at $\mathbb{E}\{\bf Y\}$.

\begin{remark}
It is beneficial to applying CV when the variance of $\kappa(x, {\bf Y})$ is small (i.e., small uncertainty) so that $\kappa(x, {\bf Y}) \Psi_{j_1}({\bf Y})\Psi_{j_2}({\bf Y})$ and $\widetilde{\kappa}(x, {\bf Y}) \Psi_{j_1}({\bf Y})\Psi_{j_2}({\bf Y})$ are highly correlated.
However, when $\kappa(x, {\bf Y})$ has large variance, the variance reduction achieved by CV may not be significant enough. In this case, it may not be worthwhile applying the CV considering the extra computational effort for estimating $\lambda_{j_1,j_2}^*$.
\end{remark}

\section{Numerical experiments}\label{sec:example}
In this section, we assess the numerical performance of the SGD-PCE solver through several one dimensional problems in order to demonstrate the efficiency and accuracy of the algorithm. 


\subsection{Model linear problem with non-homogeneous random field}\label{subsec:example-1}
We first apply the SGD-PCE algorithm to a model linear elliptic problem.
In order to make the example analytically tractable, we consider the case when the spatial 
dimension $d = 1$, that is,
\[-(\kappa(x, {\bf Y})u^{\prime}(x, {\bf Y}))^{\prime} + f(x, {\bf Y}) = 0,\qquad x \in D = \left[-l/2, l/2\right]\]
with zero deterministic boundary conditions $u(-l/2, {\bf Y}) = u(l/2, {\bf Y}) = 0$.
The random diffusivity coefficient $\kappa$ is assumed to be a log-normal random field 
\[\kappa(x, {\bf Y}) = e^{\beta V(x, {\bf Y})},\]
where $V$ is a nonlinear function of the
random vector ${\bf Y} = (A_1, \ldots, A_{n_V}, B_1, \ldots, B_{n_V})$, namely,
\[
V(x, {\bf Y}(\omega)) = \frac{1}{\sqrt{n_V}}\sum_{k=1}^{n_V} A_k(\omega) \cos\left(\frac{2\pi k x}{l}\right)
+ B_k(\omega) \sin\left(\frac{2\pi k x}{l}\right),
\]
and $A_k$ and $B_k$ are independent unit normal random variables.
The same example was considered in~\cite{field2015efficacy}. 
We can easily verify that $V(x, {\bf Y})$ is a Gaussian random field with zero mean and covariance kernel 
\[
\textrm{Cov}_V(x_1, x_2) = \frac{1}{n_V}\sum_{k=1}^{n_V} \cos\left(\frac{2\pi k(x_2-x_1)}{l}\right).
\]
The functional associated with this problem is 
\begin{equation}
E(u) = \mathbb{E}\left\{\int_D \frac{1}{2} \kappa(x, {\bf Y}) (u^{\prime}(x, {\bf Y}))^2 + f(x, {\bf Y}) u(x, {\bf Y}) \, dx\right\}
\end{equation}
and we seek for an approximated minimizer
\[u_c(x, {\bf Y}) = \sum_{i=1}^M \sum_{j=0}^N c_{ij} \phi_i(x) \Psi_j({\bf Y})\]
over the space $V_M \otimes S_N$
that solves the following problem
\[
\min_{c \in \mathbb{R}^{M(N+1)}} J(c) = \mathbb{E}
\left\{
\int_D \frac{1}{2} \kappa(x, {\bf Y}) (u_c^{\prime}(x, {\bf Y}))^2 + f(x, {\bf Y}) u_c(x, {\bf Y}) \, dx
\right\}.
\]
Recall the gradient estimator~\eqref{eqn:gradient}, an unbiased estimator of the partial derivative of $J(c)$ with respect to $c_{i,j}$, is
\[
g_{i,j}(c, {\bf Y}) = \int_D \kappa(x, {\bf Y}) u_c^{\prime}(x, {\bf Y}) \phi_i^{\prime}(x) \Psi_j({\bf Y}) + f(x, {\bf Y}) \phi_i(x) \Psi_j({\bf Y})\, dx
\]
for every $i=1,\ldots, M$ and $j = 0,\ldots, N$.
Similarly, the Hessian estimator~\eqref{eqn:Hessian} in the linear case reads
\[
h_{(i_1,j_1), (i_2, j_2)}(c, {\bf Y}) = \int_D \kappa(x, {\bf Y})\phi_{i_1}^{\prime}(x) \Psi_{j_1}({\bf Y}) \phi_{i_2}^{\prime}(x) \Psi_{j_2}({\bf Y})\, dx
\]
for every $i_1, i_2=1,\ldots, M$ and $j_1, j_2 = 0,\ldots, N$.
An important observation is that the Hessian estimator is independent of 
the vector of coefficients $c$ for the linear problem.

Before solving the minimization problem, we demonstrate the variance reduction of the CV strategy proposed in Section~\ref{subsec:control-variates}.
We are interested in comparing the standard deviations of the standard gradient estimator, the zeroth order control variates gradient estimator~\eqref{eqn:kappa-zeroth-order-approx} and the first order control variates estimator~\eqref{eqn:kappa-first-order-approx} by applying them to compute the first component of the gradient $\nabla J(c)$ (i.e., $\partial_{c_{1,0}} J(c)$) for a fixed vector of coefficients $c$. 
Note that smaller $\beta$ involved in the random field $\kappa$ leads to better approximation to the random field by \eqref{eqn:kappa-first-order-approx} since the variance of $\kappa(x, {\bf Y})$ decreases with respect to $\beta$ for all $x \in D$.
The simulation result is shown in Table~\ref{tab:beta-control-variates-gradient}.
We observe significant variance reductions achieved by the CV technique particularly when the variance of the random field $\kappa$ is small (controlled through $\beta$).
We shall see later that this variance reduction in estimating the stochastic gradient at each SGD iteration can stabilize the convergence of SGD.

\begin{table}[H]
\centering
\begin{tabular}{ c c c c c }
\hline\hline
 $\beta$ & 0.05 & 0.1 & 0.2 & 0.4\\ 
 \hline
 Without CV & $5.0161$ & $5.1200$  &  $5.4821$ & $6.9623$\\
 \hline
 $0$-th order CV  & $0.3273$  & $0.6619$  & $1.3816$ & $3.2591$ \\  
 \hline
  $1$-st order CV & $0.0149$  & $0.0849$  & $0.2488$ & $1.1338$ \\  
 \hline\hline
\end{tabular}
\caption{Standard deviations of the gradient estimators with CV, the first order control variates gradient estimator and the standard gradient estimator for estimating $\partial_{c_{1,0}} J(c)$ (where $c$ is fixed).
In this computation, the stochastic germ is ${\bf Y} = (A_1, A_2, B_1, B_2)$ so that $n_V = 2$.
The PC expansion is truncated up to the order $p = 3$ so that the total number of PC basis size is $N = 35$. The FEM space consists of $M = 10$ linear elements. The standard deviations are estimated over $10^5$ Monte Carlo realizations.}
\label{tab:beta-control-variates-gradient}
\end{table}

Now, we aim to solve the minimization problem in the homogeneous case, i.e., $f(x, {\bf Y}) = 0$.
In this case, the solution is trivial with the unique minimizer $c_{ij}^* = 0$ for all
$i=1, \ldots, M$ and $j = 0, \ldots, N$ and the minimum value
$J(c^*) = 0$.
Therefore, this is an ideal benchmark for testing the accuracy of the SGD-PCE algorithm.
The computed minimum values corresponding to various learning rates
are shown in Table~\ref{tab:learning-rate-vs-energy} with fixed $\beta = 0.1$.
We compare the minimum value computed using CV gradient estimation with that computed 
without using CV.
For the same learning rate $\eta_n$, the result computed with CV is always more accurate than that computed without CV. In particular, when the learning rate is fairly large, the result computed with CV is extremely accurate while that without CV does not even converge!
We confirm this observation by plotting the convergence behavior (in log-log scale) of energies $J(c_n)$ in Figure~\ref{fig:iteration-vs-energy}.
Since the gradient estimation with CV is less noisy than that without CV, 
we expect that the trajectory of the former should be less fluctuating than that of the latter when in the stochastic regime of SGD.
The right panel of Figure~\ref{fig:iteration-vs-energy} confirms this expectation. 
Therefore, although the convergence rate remains roughly the same, SGD with CV does stablize the convergence at the initial stage.
We also plot the divergent behavior obtained from the first order SGD (i.e., without Hessian rescaling) in the left panel of Figure~\ref{fig:iteration-vs-energy}, which suggests that incorporating the second order information is necessary.

\begin{table}[H]
\centering
\begin{tabular}{ c c c c c c}
\hline\hline
 Learning rate & $\frac{1}{n+2}$ & $\frac{2}{n+2}$ & $\frac{5}{n+2}$ & $\frac{10}{n+2}$ & $\frac{100}{n+2}$\\ 
 \hline
  $1$-st order CV  & $6.1\times 10^{-3}$  & $3.0\times 10^{-7}$  & $4.8\times 10^{-18}$ & $3.4\times 10^{-26}$ & $5.8\times 10^{-49}$\\  
 \hline
 Without CV & $1.6\times 10^{-2}$ & $1.1\times 10^{-5}$  &  $6.8\times 10^{-11}$ & $2.1\times 10^{-13}$ & NA\\
 \hline\hline
\end{tabular}
\caption{The minimum energies $J(c^*)$ computed by SGD-PCE under different
learning rate $\eta_n$ with $\beta = 0.1$.
Both the results computed with first order control variates and without control variates are shown to illustrate the advantage
of incorporating the control variates gradient estimation.
Results computed with and without control variates gradient estimations 
The total number of PC basis size is $N = 35$ and the total number of FEM basis is $M = 50$. The results are based on $N_{\textrm{sgd}}=500$ SGD iterations with mini-batch size $N_g = 128$ for gradient estimation and mini-batch size $N_h = 64$ for Hessian estimation.}
\label{tab:learning-rate-vs-energy}
\end{table}

\begin{figure}[H]
    \centering
    \includegraphics[width=130mm]{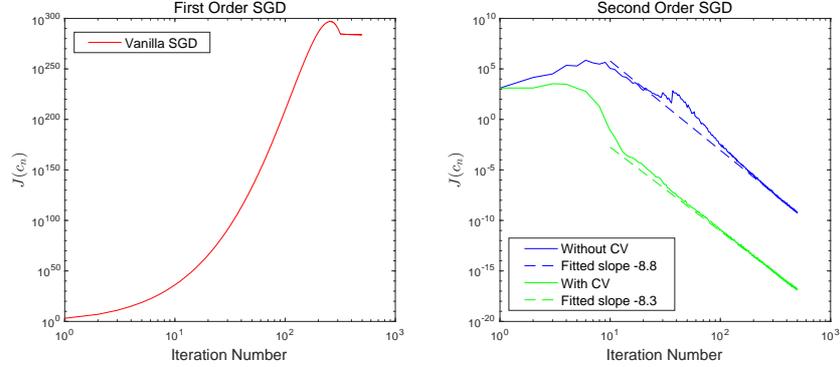}
    \caption{Left: log-log plot of $J(c_n)$ obtained from the first order SGD. Right: log-log plot of $J(c_n)$ obtained from the second order SGD with and without CV. Setup of the simulation is same as that in Table~\ref{tab:learning-rate-vs-energy} with learning rate $\eta_n = 5(n+2)^{-1}$.}
    \label{fig:iteration-vs-energy}
\end{figure}

Finally, we apply the SGD-PCE to the case of non-homogeneous boundary
\begin{equation*}
\begin{split}
&-(\kappa(x, {\bf Y}) u^{\prime}(x, {\bf Y}))^{\prime}  = 0 \qquad x \in D =  [-l/2, l/2]\\
&u(-l/2, {\bf Y}) = 0, u(l/2, {\bf Y}) = 1.
\end{split}
\end{equation*}
We test the accuracy of the approximated solution $u_c$ by examining its joint distribution at $x=-4$ and $x = 2$, i.e., 
\[
F_c(y_1, y_2) = \mathbb{P}\{u_c(-4, {\bf Y}) \leq y_1, u_c(2, {\bf Y}) \leq y_2\}.
\]
In Figure~\ref{fig:CDF-landscape}, we plot of the 
the cumulative distribution function (CDF) of the exact solution $u$ over a range of $y_1$ and $y_2$ (left panel) and the error between the exact CDF and the approximated CDF (right panel).
The right panel shows that the error varies from $-2\%$ to $+7\%$ depending on the location $(y_1, y_2)$ where we evaluate the CDFs.


\begin{figure}[H]
\centering
\subfigure{\label{fig:a}\includegraphics[width=65mm]{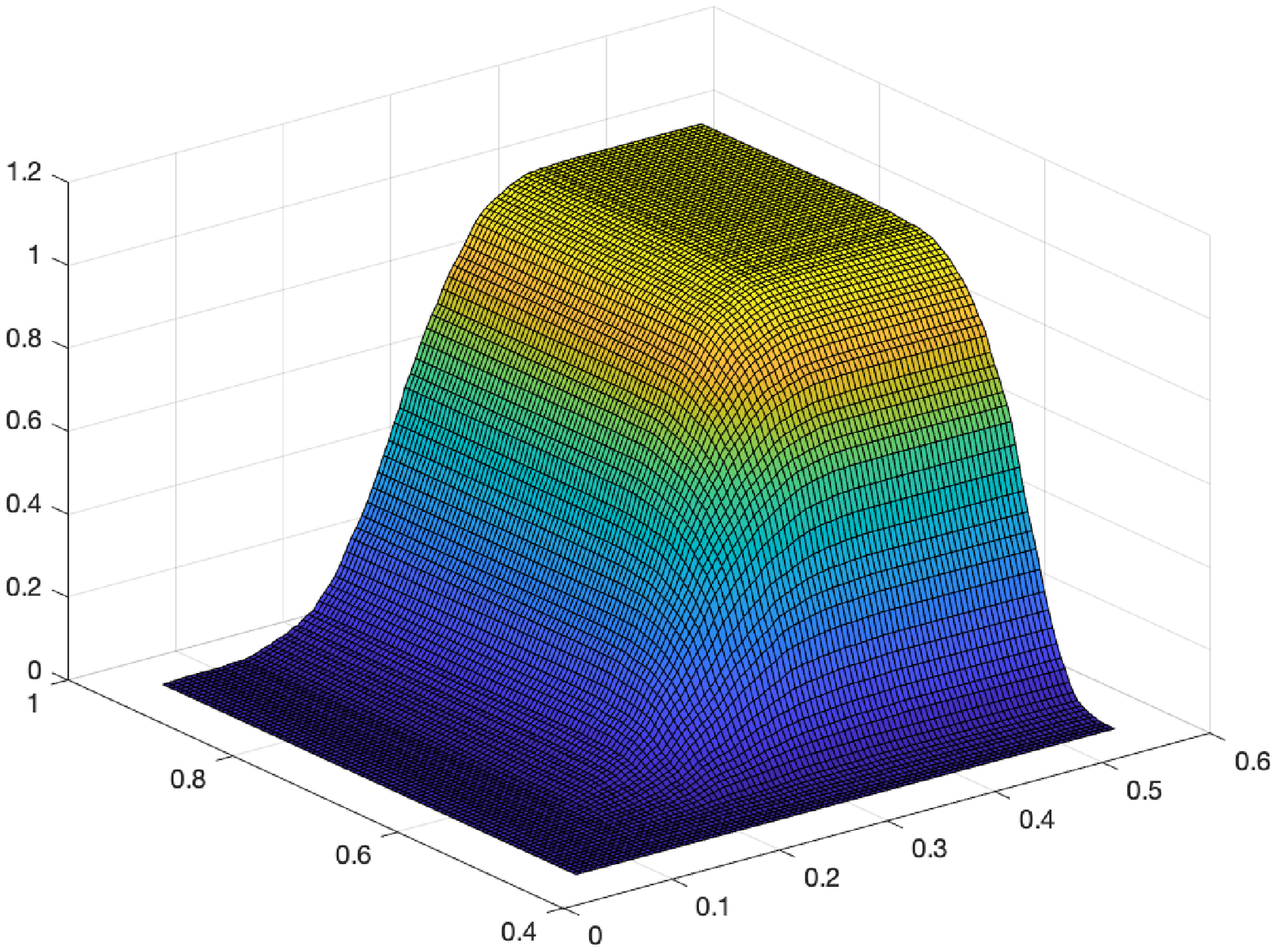}}
\subfigure{\label{fig:b}\includegraphics[width=65mm]{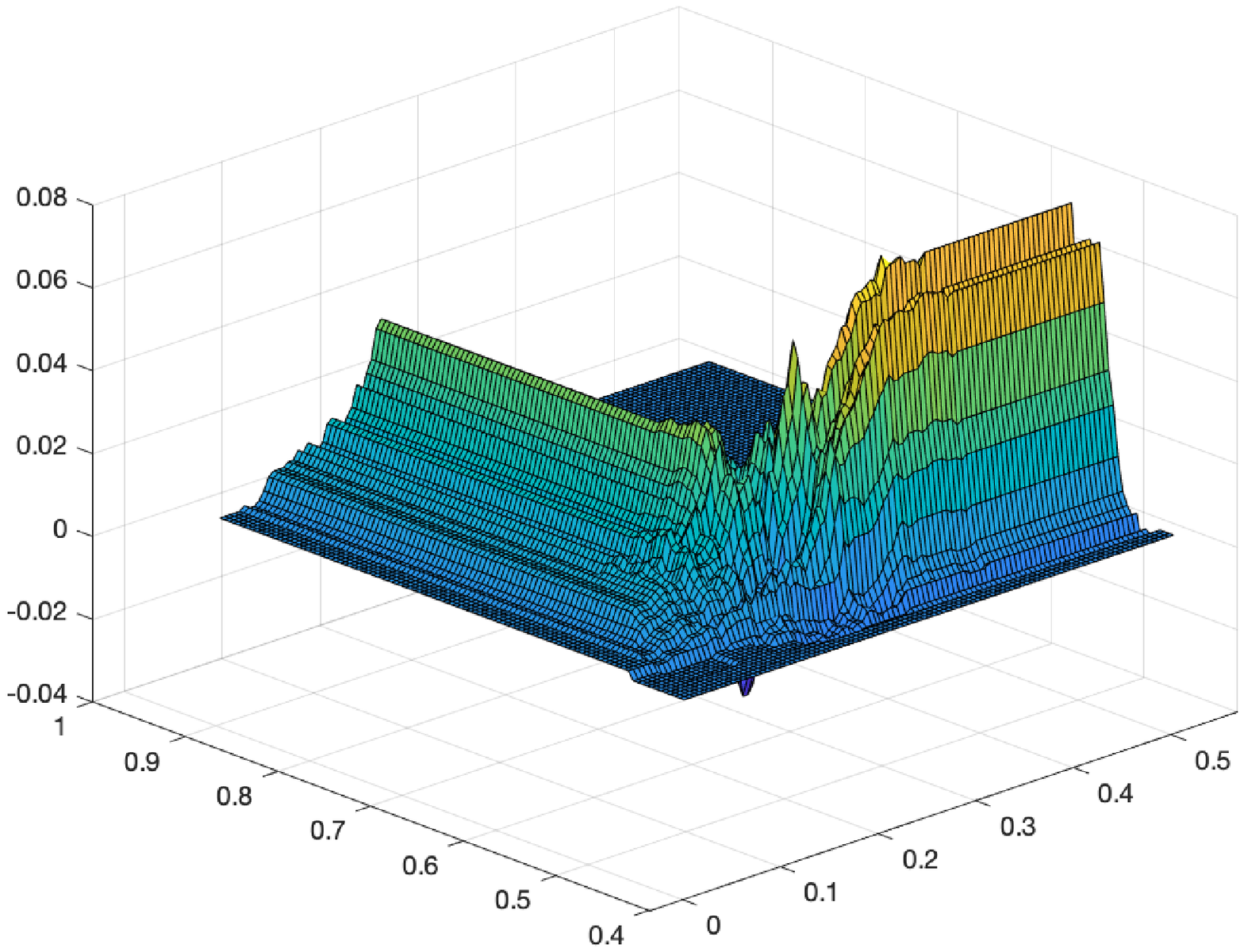}}
\caption{Left: the exact landscape of $F(y_1, y_2) = \mathbb{P}\{u(-4, {\bf Y}) \leq y_1, u(2, {\bf Y}) \leq y_2\}$; Right: the error between the exact CDF and the approximated CDF computed from SGD-PCE, i.e., $F(y_1, y_2) - F_c(y_1, y_2)$.
The learning rate is chosen to be $\eta = 10/n$.}
\label{fig:CDF-landscape}
\end{figure}


\subsection{Model semilinear problem with homogeneous random field}
In this example, we add a nonlinear term to the problem but make the random field $\kappa$ homogeneous (i.e., $\kappa(x, {\bf Y}) = \kappa({\bf Y})$), namely,
\begin{equation}
\begin{split}
-(\kappa({\bf Y}) u^{\prime}(x, {\bf Y}))^{\prime} + f(x, {\bf Y}) + \sin(u(x, {\bf Y})) &= 0 \qquad x \in D = [-{l}/{2}, {l}/{2}],\\
u(-l/2, {\bf Y}) = u(l/2, {\bf Y}) &= 0,
\end{split}
\end{equation}
where the homogeneous random field 
$
\kappa({\bf Y}) = \exp(0.2(Y_1 + Y_2))
$ 
and 
\[
f(x, {\bf Y}) = \pi^2 \sin(\pi x) + \sin\left(\frac{\sin(\pi x)}{\kappa({\bf Y})}\right).
\]
Note that the problem has an exact solution 
\[u(x, {\bf Y}) = \frac{\sin(\pi x)}{\kappa({\bf Y})}\]
which will be used to assess the accuracy of the SGD-PCE solver. 
Given $M$ FEM bases and $N+1$ PC bases, the function $J(c)$ we shall minimize is
\[
J(c) = \mathbb{E}\left\{\int_D \frac{1}{2}\kappa({\bf Y}) (u_c^{\prime}(x, {\bf Y}))^2 + f(x, {\bf Y}) u_c(x, {\bf Y}) - 
\cos(u_c(x, {\bf Y}))  \, dx\right\},
\]
where $u_c(x, {\bf Y})$ is the truncated generalized PC expansion
\[u_c(x, {\bf Y}) = \sum_{i=1}^M \sum_{j=0}^N c_{ij} \phi_i(x)\Psi_j({\bf Y}).\]

\begin{table}[H]
\centering
\begin{tabular}{ c c c c c }
\hline
\hline
 PCE Order & $p=0$ & $p=1$ & $p=2$ & $p=3$\\ 
 \hline
$J(c^*)$ & $-42.7273$  & $-44.9673$  & $-45.0693$ & $-45.1252$ \\  
 \hline 
$L^2$ Error & $9.10\times 10^{-2}$ & $4.30 \times 10^{-3}$ & $3.01 \times 10^{-4}$ & $1.04 \times 10^{-4}$\\
 \hline
 \hline
\end{tabular}
\caption{
The approximated minimum $J(c^*)$ and the $L^2$ error with respect to the order of PC expansion.
The exact minimum value, $E(u^*) = -45.4040$, is computed from Monte Carlo simulation with $10^5$ samples.
The approximate minimum value is computed using SGD-PCE with mini-batch size 
$N_g = N_h = 100$, learning rate $\eta_n = 10/n$ and total number of iterations $N_{\textrm{sgd}}=1000$. 
The physical space $V$ is approximated by $M = 100$ FEM bases. 
The $L^2$ error is evaluated at $x=0.5$.}
\label{tab:PCEorder-vs-energy}
\end{table}

Table~\ref{tab:PCEorder-vs-energy} shows the approximated minimum value $J(c^*)$ and $L^2$ error 
\[\mathbb{E}[(u(x, {\bf Y}) - u_c(x, {\bf Y}))^2]\] 
with respect to the order of PC expansion. The result is consistent with the fact that higher order approximation to the stochastic space leads to smaller error.
Recall that the exact solution $u^*$ is known and hence the exact functional value $E(u^*)$
can be computed as well.
This allows us to plot the convergence of $J(c_n)$ to the true minimum $E(u^*)$.
Figure~\ref{fig:SGD-learning-rate} demonstrates this convergent behavior of SGD with various learning rates $\eta_n$.
An interesting observation is that the difference $J(c_n) - E(u^*)$ dramatically jumps towards zero and then converges slowly to the limit with a steady rate.
However, with slow learning rate, the difference jumps only to a region far away from zero and then enters into the slow convergence phase in a highly noisy manner. In contrast, with faster learning rate, SGD is able to jump quickly to a neighborhood of zero only after a few initial iterations.
This phenomenon suggests that we should always choose a learning rate as fast as possible so long as it does not exceed the theoretical threshold predicted by Theorem~\ref{thm:main-thm}.

\begin{figure}[h]
\centering
\includegraphics[width=130mm]{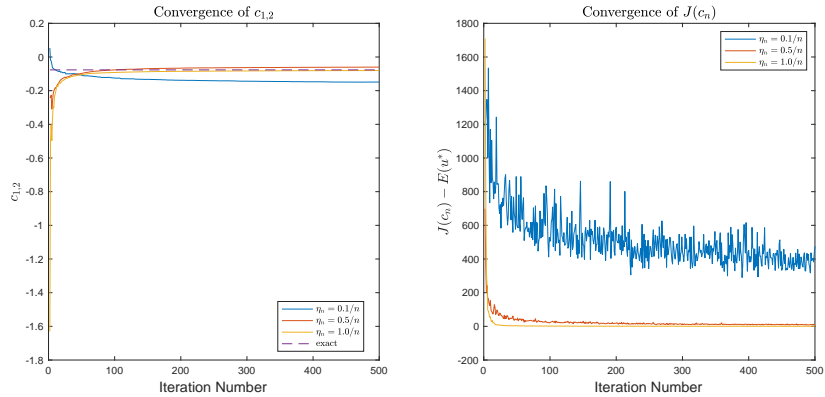}
\caption{
Left: convergence of the coefficient $c_{1,2}$. Right: convergence of $J(c_n)$.
The computation uses SGD-PCE with min-batch size 
$N_g = N_h = 100$. The physical space $V$ is approximated by $M = 100$ FEM basis. 
The stochastic space $S$ is approximated by PC expansion of $2$ random variables up to order $p = 3$ so the total number of PC basis is $N = 10$.}
\label{fig:SGD-learning-rate}
\end{figure}

Finally, we assess the distribution of the approximated solution $u_c(x, {\bf Y})$.
To this end, we compare the CDF of the approximate solution $u_c$ and that of the exact solution $u$ at the point $x=0.5$.
The CDFs obtained by Monte Carlo simulation over $10^5$ samples are plotted in Figure~\ref{fig:CDF}. 

\begin{figure}[h]
\centering
\includegraphics[width=70mm]{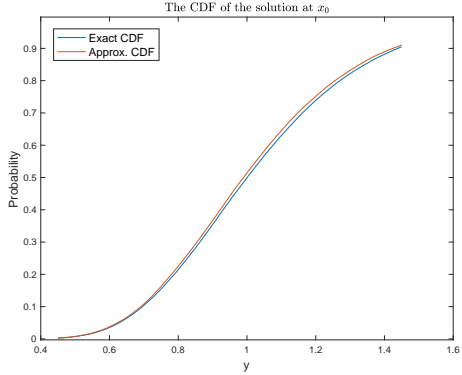}
\caption{The CDF of the solution at $x = 0.5$. The computation uses SGD-PCE with batch size $N_g = N_h = 100$ with a diminishing learning rate $\eta_n = 10/n$. 
The stochastic space $S$ is approximated by PC expansion of $2$ random variables up to 
order $p = 3$.
The physical space $V$ is approximated by $M = 100$ FEM basis.}
\label{fig:CDF}
\end{figure}

\subsection{Model semilinear problem with non-homogeneous random field}
Finally, we study a semi-linear problem with non-homogeneous random field, i.e.,  
\begin{equation*}
\begin{split}
-(\kappa(x, {\bf Y}) u^{\prime}(x, {\bf Y}))^{\prime} + \sin(u(x, {\bf Y})) &= 0, \qquad x \in D  = [-l/2, l/2]\\
u(-l/2, {\bf Y}) = u(l/2, {\bf Y}) &= 0,
\end{split}
\end{equation*}
where $\kappa(x, {\bf Y})$ is the same log-normal random field as in the Section~\ref{subsec:example-1}. 
Note that the exact solution is $u^*(x, {\bf Y}) =0$ and hence the exact minimum value is $E(u^*) = 12$. 
Over the finite dimensional space $V_M \otimes S_N$, 
the functional is
\[
J(c) = \mathbb{E}\left\{\int_D \frac{1}{2}\kappa(x, {\bf Y}) (u_c^{\prime}(x, {\bf Y}))^2 - \cos(u_c(x, {\bf Y}))  \, dx\right\}.
\]
The gradient estimator $g(x, {\bf Y})$ consists of two parts (see~\eqref{eqn:gradient})
\begin{equation*}
\begin{split}
g_{i,j}^1(c, {\bf Y})
&=
\int_D \kappa(x, {\bf Y}) u_c^{\prime}(x, {\bf Y}) \phi_i^{\prime}(x) \Psi_j({\bf Y}) \, dx,\\
g_{i,j}^2(c, {\bf Y})
&=
\int_D \sin(u_c(x, {\bf Y}))\phi_i(x) \Psi_j({\bf Y})\,dx
\end{split}
\end{equation*}
for all $i = 1, \cdots, M$ and $j=0, \cdots, N$.
Similarly, the Hessian estimator also has two parts
\begin{equation*}
    \begin{split}
       h_{(i_1,j_1), (i_2, j_2)}^1(c, {\bf Y}) &= \int_D \kappa(x, {\bf Y})\phi_{i_1}^{\prime}(x) \Psi_{j_1}({\bf Y}) \phi_{i_2}^{\prime}(x) \Psi_{j_2}({\bf Y})\, dx \\
       h_{(i_1,j_1), (i_2, j_2)}^2(c, {\bf Y}) &= \int_D \cos(u_c(x, {\bf Y}))\phi_{i_1}(x) \Psi_{j_1}({\bf Y}) \phi_{i_2}(x) \Psi_{j_2}({\bf Y})\, dx
    \end{split}
\end{equation*}
for all $i_1, i_2 = 1, \cdots, M$ and $j_1, j_2 =0, \cdots, N$.
Observe that the nonlinear part $h^2(c, {\bf Y})$ depends on the coefficient $c$ whereas the linear part $h^1(c, {\bf Y})$ does not.
Hence, $h^2(c, {\bf Y})$ can be extremely noisy at the initial stage when SGD is still in its stochastic regime. 
The noisy estimation of $h^2(c, {\bf Y})$ in turn may have a detrimental rather than beneficial effect to guide the next search direction of SGD.
In contrast, the linear part $h^1(c, {\bf Y})$, although only contains partial second order information, is immune from the noisy updates of $c_n$ at the initial stage of SGD. 
This observation suggests that we can simply utilize the linear part of the Hessian estimator at the initial stage of SGD and incorporate the nonlinear part only after SGD gets stabilized.
In Figure~\ref{fig:full_vs_adaptive_hessian}, we demonstrate the fast convergence behavior of SGD-PCE by incorporating the nonlinear part Hessian $h^2(c, {\bf Y})$ after $100$ iterations. 
In comparison, with the same learning rate, a naive use of the full Hessian information even does not lead to a converged result.

\begin{figure}[h]
\centering
\includegraphics[width=0.5\textwidth]{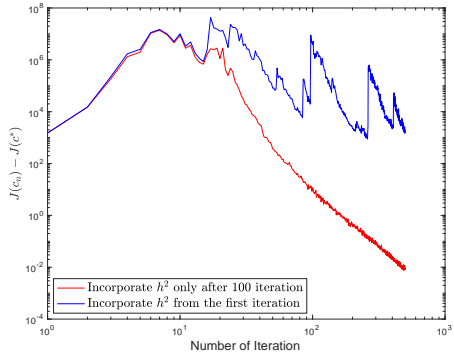}
\caption{Convergence behavior of $J(c_n)$ with the nonlinear part of Hessian incorporated at different stages (log-log scale). The SGD uses a diminishing learning rate $5(n+2)^{-1}$. The stochastic space is approximated by PC expansion of $4$ random variables up to order $p = 3$.
The physical space $V$ is approximated by $M = 50$ FEM basis.}
\label{fig:full_vs_adaptive_hessian}
\end{figure}

Finally, Figure~\ref{fig:grad-hessian} shows the effect of mini-batch sizes $N_g$ and $N_h$ on the convergence of SGD-PCE. The nonlinear part of Hessian is only incorporated after $100$ iterations. Recall that larger $\beta$ corresponds to
larger variance of the random field $\kappa$. 
When $\beta = 0.3$, mini-batch of size $128$ for gradient and $64$ for Hessian are not enough for SGD to converge in $500$ iterations. 
However, an increase of either $N_g$ or $N_h$ helps overcome the ill-conditioning issue of SGD.
When $\beta = 0.4$, $128$ mini-batch samples for gradient estimation are not sufficient even when we increase $N_h$ to $128$. 
However, the algorithm converges after we increase $N_g$ to $256$, which suggests that SGD iteration is more tolerant to noise in the Hessian estimation than it is to the gradient estimation.

\begin{figure}[h]
\centering
\includegraphics[width=0.8\textwidth]{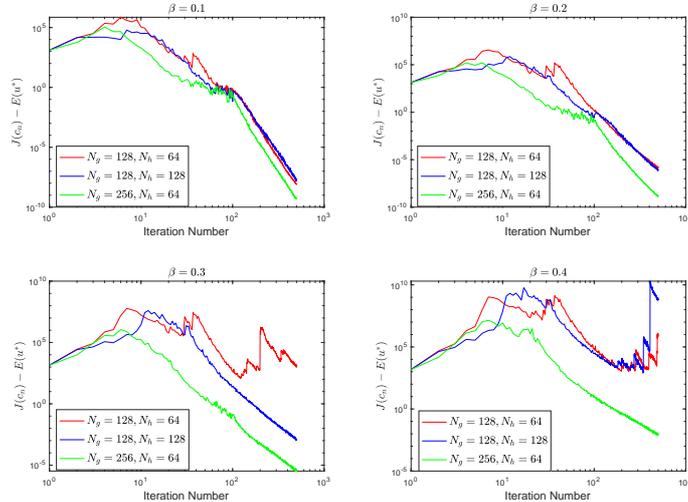}
\caption{Convergence of $J(c_n)$ with various mini-batch sizes for gradient and Hessian (log-log scale). The SGD uses a diminishing learning rate $5(n+2)^{-1}$. The stochastic space $S$ is approximated by PC expansion of $4$ random variables up to order $p = 3$.
The physical space $V$ is approximated by $M = 50$ FEM bases.}
\label{fig:grad-hessian}
\end{figure}

\section*{Summary and Conclusion}
\label{sec:summary}
We have presented a variational framework for solving semilinear PDEs
with random coefficients. The framework relies on the direct methods
of variational calculus to recast the stochastic PDE as a stochastic
minimization problem, which can then be solved by SGD over
finite-dimensional subspaces.  Our variational framework offers key
advantages over traditional approaches based on weak formulations.
First and foremost, from the theoretical standpoint, the direct
methods of variational calculus automatically ensure weak convergence
of the numerical solutions obtained under this framework. Second, our
framework is able to take advantage of the countless SGD algorithms
developed during the rapid ascend of machine learning in the last
decade. Finally, the variational approach is well known to retain
certain structures of the original problem and hence is particularly
advantageous when applied to structure preserving problems.  Based on
this framework, we have proposed an SGD-PCE method utilizing the
general PC expansion to approximate the stochastic space. By taking
advantage of the special structure of the optimization problem derived
from the PC expansion, we are able to design a version of SGD
algorithm that finds the minimizer in an efficient way. We emphasize
that, under our framework, other existing finite dimensional
approximation techniques can be readily utilized in the same spirit.
Despite the various advantages mentioned above, we are also aware of
some weaknesses of our variational framework. For instance, the
SGD-PCE can be computationally very costly, especially when the
dimensionality of the approximation space $V_M\otimes S_N$ is high.
Furthermore, a successful application of SGD in our framework requires
careful tuning of multiple hyper-parameters, e.g. the learning rate,
iteration number and mini-batch size. Often, such tuning is ad-hoc and
hence can be very challenging for complex problems. In order to
overcome these difficulties, it may be beneficial to apply more
advanced variants of SGD such as the adaptive momentum (Adam)~\cite{kingma2014adam}
algorithm capable of adaptively adjusting the learning rate and
efficiently dealing with problems involving large set of
parameters. This will be the focus of our future work.

\section*{Acknowledgments}
\label{sec:acknowledgements}
The authors would like to thank Petr Plech\'a\v{c} and Gideon Simpson for fruitful
discussions.
The research of T.W. was sponsored by the CCDC Army Research Laboratory and was accomplished under Cooperative Agreement Number W911NF-16-2-0190.
The views and conclusions contained in this document are those of the authors and should not be interpreted as representing the official policies, either expressed or implied, of the Army Research Laboratory or the U.S. Government. The U.S. Government is authorized to reproduce and distribute reprints for Government purposes notwithstanding any copyright notation herein.

\bibliography{SGD_PCE}

\end{document}